\documentclass[12pt]{article}
\usepackage{amssymb,amsmath}
\usepackage{graphicx,float,color}
\usepackage[encapsulated]{CJK}

\usepackage{geometry} 
\usepackage{amsthm}
\usepackage{array} 
\usepackage{paralist} 
\usepackage{verbatim} 
\usepackage{subfig} 
\usepackage{amsmath}
\usepackage{amssymb}
\usepackage{amsfonts}
\usepackage[all]{xy}
\newtheorem{theorem}{Theorem}[section]
\newtheorem {lemma}{Lemma}[section]
\newtheorem {proposition}{Proposition}[section]
\newtheorem {definition}{Definition}[section]
\newtheorem {example}{Example}[section]
\newtheorem {corollary}{Corollary}[section]
\newtheorem{remark}{Remark}
\newcommand{\D}{\displaystyle}

\usepackage{sectsty}
\allsectionsfont{\sffamily\mdseries\upshape} 

\title{Backward stochastic differential equations with stopping time as time horizon}
\author{ Munchol Kim, Cholkyu Pak}
\date{}

\begin{document}
\maketitle

\centerline{\small Department of Mathematics, Kim Il Sung
University, Pyongyang, D.P.R.Korea}

\begin{abstract} In this paper, we introduce a new method for study on backward stochastic differential equations with stopping time as time horizon.
 And using this, we show that some results on backward stochastic differential equations with constant time horizon are generalized to the case of random time horizon.
\end{abstract}

\vskip0.6cm\noindent {\bf Keywords}: BSDE(backward stochastic
differential equation), random time horizon, measure solution

\section{Introduction}
Backward stochastic differential equations(BSDEs) were first
introduced by E. Paradoux and S. Peng in 1990 [2]. Since then
BSDE¡¯s have been widely used in mathematical finance and partial
differential equations(PDEs). Many pricing problems can be written
in terms of linear BSDEs, or non-linear BSDEs when portfolios
constraints are taken into account as in El Karoui, Peng and
Quenez [10]. And numerous results show the intimate relationship
between BSDEs and PDEs, which suggests that existence and
uniqueness results which can be obtained on one side should have
their counterparts on the other side.([13,14, 15])

Many mathematicians have worked to improve the
existence/uniqueness conditions of a solution for BSDEs in
connection with the specific applications.([6, 7, 12])

\noindent Most of those works are concerned with the case of
constant time horizon.

But in many applications we encounter the case of random time
horizon. For example M.Marcus and L. Veron show that the solution
to PDE : $-\Delta u+u\lvert u\rvert^q=0$ are related to the BSDE :
$Y_t=\xi-\D\int_t^\tau Y_r\lvert Y_r\rvert^q dr-\D\int_t^\tau
Z_rdW_r$. Here the time horizon $\tau$ is a stopping time. And in
finance, there are many cases when the time horizon is not
constant but random. There are many research papers about the case
of random time horizon in connections with applications.([9,11])

But those works on the case of random time horizon is limited to
their specific cases and in many papers additional efforts are put
to generalize their results to the case of random time horizon. In
this paper, we suggest a general map from one side to the other
side in section 2(Theorem 2.1) and show some results using this as
examples in section 3. We note that our work is devoted to the
case of stopping time and it will meet the most of applications.

\section{Map from the case of constant time horizon to the case of random time horizon}

In this section, we suggest a general map from the case of
constant time horizon to the case of random time horizon.

Let $(W_t)_{t\ge0}$ be a standard one dimensional Brownian motion
on a complete probability space $(\Omega, \mathcal{F}_t,
\mathbb{P}, \mathbb{F}=(\mathcal{F}_t)_{t\ge0})$.

We will use the following notations:

\vskip 0.5cm
\indent\indent$\mathcal{L}_2^{loc}:=\{\Phi=\{\Phi(t)\}_{t\ge0}$:
real valued measurable processes adapted to the filtration
$(\mathcal{F}_t)_{t\ge0}$ such that $\forall T>0,
\D\int_0^T\Phi^2(t,\omega)dt<\infty\quad a.s.\}$ \vskip 0.2cm
\indent\indent$\mathcal{M}_2^{loc}:=\{X=(X(t))_{t\ge0}$: locally
square integrable $(\mathcal{F}_t)-$ martingales such that
$X_0=0\quad a.s.\}$ \vskip 0.2cm
\indent\indent$\mathcal{M}_2^{c,loc}:=\{X \in \mathcal{M}_2^{loc}
: t\rightarrow X_t$ is continuous $a.s.\}$

\vskip 0.5cm
\begin{definition}[Time Change]
Let $(\Phi(t,\omega))_{t\ge0}$ be a continuous real valued process
on $(\Omega, \mathcal{F}, \mathbb{P}, \mathbb{F})$ which satisfies
following conditions.
\vskip 0.2cm

(i) $(\Phi(t,\omega))$ is adapted to $(\mathcal{F}_t)_{t\ge0}$

(ii) $t\rightarrow \phi(t)$ is strictly increasing $a.s.$

(iii) $\D\lim_{t\rightarrow\infty}\phi_t=\infty\quad a.s.$

(iv) $\phi(0)=0\quad a.s.$ \vskip 0.2cm Then we call $\phi$ {\bf
time change} on $(\Omega, \mathcal{F}, \mathbb{P}, \mathbb{F})$.
\end{definition}

\vskip0.5cm
\begin{example}
If $\tau$ is a $(\mathcal{F}_t)-$stopping time,
$\phi(t,\omega)=\dfrac{t}{\tau},\;t\in [0,\infty]$ is a time
change on $(\Omega, \mathcal{F}, \mathbb{P}, \mathbb{F})$.
\end{example}

\vskip0.5cm From the definition, if $\phi$ is a time change, there
exist an inverse function almost surely and we have an inverse
process $\phi^{-1}(t,\omega)$. Then $\phi^{-1}_t(\omega)$ is a
stopping time for all $t$ because $\{\phi^{-1}_t(\omega)\le
s\}=\{t \le \phi_s(\omega)\}\in\mathcal{F}_s$.

\noindent And if $X=(X_t)$ is a continuous $(\mathcal{F}_t)-$
adapted process, then $T^\phi X=((T^\phi X)(t))=(X(\phi _t^{-1}))$
is a continuous $(\mathcal{F}_{\phi_t^{-1}})-$ adapted process.
Given a process of time change $\phi$, we define a new reference
family $(\tilde{\mathcal{F}}_t)$ by
$\widetilde{\mathcal{F}}_t=(\mathcal{F}_{\phi_t^{-1}}),\;t\in[0,\infty)$
. The class $\mathcal{M}_2^{c,loc}$ with respect to
$\widetilde{\mathcal{F}}_t$ is denoted by
$\mathcal{\widetilde{M}}_2^{c,loc}$. By Doob¡¯s optional sampling
theorem if $M\in\mathcal{M}_2^{c,loc}$ then
$\widetilde{M}\in\mathcal{M}_2^{c,loc}$. And
\begin{displaymath}
X,Y\in\mathcal{M}_2^{c,loc}\Rightarrow<X_{\phi^{-1}},Y_{\phi^{-1}}>=<X,Y>_{\phi^{-1}}
\end{displaymath}

So we can see that the characters of process are preserved in time
change.

\vskip0.5cm
\begin{lemma} If $\phi(t,\omega)$ is a time change on
$(\Omega, \mathcal{F}, \mathbb{P}, \mathbb{F})$,
$\phi^{-1}(t,\omega)$ is a time change on $(\Omega, \mathcal{F},
\mathbb{P},
\widetilde{\mathbb{F}}=(\widetilde{\mathcal{F}}_t)_{t\ge0})$.
\end{lemma}

\vskip0.5cm Now we show how the stochastic integral changes in
time change.

\vskip0.5cm
\begin{proposition}[{[1]}]
Let $M\in\mathcal{M}_2^{loc}$ and $\Phi\in\mathcal{L}_2^{loc}(M)$.

\noindent Then $X=I^M(\Phi):=\D\int_0^{.}\Phi(s)dM_s$ is
characterized as the unique $X\in\mathcal{M}_2^{loc}$ such that

\begin{displaymath}
<X,N>(t)=\int_0^{t}\Phi(u)d<M,N>(u)
\end{displaymath}
 for every $N\in\mathcal{M}_2^{loc}$ and all $t\ge0$.
\end{proposition}

\vskip0.5cm

\begin{lemma}
Let $\xi$ be a nonnegative random variable on $(\Omega,
\mathcal{F}, \mathbb{P}, \mathbb{F})$, $\phi(t,\omega)$ be a
$C^1$ time change on $(\Omega, \mathcal{F}, \mathbb{P},
\mathbb{F})$ and $(X_t)_{t\ge0}\in\mathcal{L}_2^{loc}(M)$.

\noindent Then
$(\widetilde{X}_t)_{t\ge0}:=(X_{\phi^{-1}(t)})_{t\ge0}\in\mathcal{L}_2^{loc}(\widetilde{M})$
and

\begin{displaymath}
\int_0^{\xi}X_sdM_s=\int_0^{\phi(\xi)}\widetilde{X}_sd\widetilde{M}_s\quad
a.s.
\end{displaymath}
\end{lemma}
\begin{proof}
As we mentioned above, $M\in\mathcal{M}_2^{loc}$ implies that
$\widetilde{M}\in\mathcal{\widetilde{M}}_2^{loc}$ and
\begin{displaymath}
\widetilde{I}:=I^{\widetilde{M}}(\widetilde{X}_s)=\int_0^{\phi(\xi)}\widetilde{X}_sd\widetilde{M}_s\in\widetilde{M}_2^{loc}
\end{displaymath}
\noindent And

\vskip 0.2cm
\begin{math}
   <\widetilde{I},\widetilde{N}>(t)=<I,N>(\phi^{-1}(t))=\\
    \textrm{\qquad\qquad}=\D\int_0^{\phi^{-1}(t)}X_sd<M,N>(s)=\\
    \textrm{\qquad\qquad}=\D\int_0^t X_{\phi^{-1}(u)}d<M,N>(\phi^{-1}(u))=\D\int_0^t
    \widetilde{X}_ud<\widetilde{M},\widetilde{N}>(u)\\
\end{math}

\vskip 0.1cm
 for all $\widetilde{N}\in\widetilde{M}_2^{loc}$ and $t\ge0$.
\vskip 0.2cm So from Proposition 2.1 we have the result.
\end{proof}

\vskip0.5cm
\begin{corollary} Let $\eta$ be another nonnegative random
variable on $(\Omega, \mathcal{F}, \mathbb{P}, \mathbb{F})$

\noindent Then we have
\begin{displaymath}
\int_\eta^{\xi}X_sdM_s=\int_{\phi(\eta)}^{\phi(\xi)}\widetilde{X}_sd\widetilde{M}_s\quad
a.s.
\end{displaymath}
\end{corollary}

\vskip 0.5cm Now we state our main result which shows the
relationship between the case of constant time horizon and the
case of random time horizon.

\begin{theorem}
Let $(Y_t, Z_t)$ be a solution of the following BSDE defined on
$(\Omega, \mathcal{F}, \mathbb{P}, (\widetilde{\mathcal{F}}_t))$

\begin{displaymath}
 \quad Y_t=\xi-\int_t^\tau Z_sdW_s+\int_t^\tau f(s, Y_s, Z_s)ds,
 \quad 0\le t\le \tau\qquad(2.1)
\end{displaymath}

where $\tau$ is $(\mathcal{F}_t)_{t\ge0}-$stopping time such that
 $0\le\tau<\infty\quad a.s.$ and $\xi\in\mathcal{F}_\tau$.

Then there exist a $C^1$ time change $\phi(t,\omega)$ on $(\Omega,
\mathcal{F}, \mathbb{P}, (\mathcal{F}_t))$ such that

\begin{displaymath}
(y_t, z_t):=(Y_{\phi^{-1}(t)},
Z_{\phi^{-1}(t)}\phi'(t)^{-\frac{1}{2}})
\end{displaymath}
is a solution of the following BSDE defined on $(\Omega,
\mathcal{F}, \mathbb{P}, (\widetilde{\mathcal{F}}_t))$.

\begin{displaymath}
 \quad y_t=\xi-\int_t^1 z_sd\widetilde{W}_s+\int_t^1 f(\phi^{-1}(s), y_s, z_s\phi'(s)^{\frac{1}{2}})\phi'(s)^{-1}ds,
 \quad 0\le t\le 1\qquad(2.2)
\end{displaymath}

\noindent where $(\widetilde{W}_t)_{t\ge0}$ is a Wiener process on
$(\Omega, \mathcal{F}, \mathbb{P}, (\widetilde{\mathcal{F}}_t))$.
\end{theorem}

\begin{proof}
Let's take a $C^1$ time change on $(\Omega, \mathcal{F},
\mathbb{P}, (\mathcal{F}_t))$ such that $\phi(\tau,\omega)=1$.(See
Example 2.1)

Now we'll prove the result for this $\phi$.

From the definition
\begin{displaymath}
y_t=Y_{\phi^{-1}(t)}=\xi-\int_{\phi^{-1}(t)}^\tau
Z_sdW_s+\int_{\phi^{-1}(t)}^\tau f(s,Y_s,Z_s)ds
\end{displaymath}

And from Lemma 2.2
\begin{displaymath}
\int_{\phi^{-1}(t)}^\tau Z_sdW_s=
\int_{\phi(\phi^{-1}(t))}^{\phi(\tau)}
Z_{\phi^{-1}(s)}dW_{\phi^{-1}(s)}=\int_t^1Z_{\phi^{-1}(s)}dW_{\phi^{-1}(s)}
\end{displaymath}

and it's clear that

\begin{displaymath}
\int_{\phi^{-1}(t)}^\tau
f(s,Y_s,Z_s)ds=\int_t^1f(\phi^{-1}(s),Y_{\phi^{-1}(s)},Z_{\phi^{-1}(s)})d\phi^{-1}(s)
\end{displaymath}

So we have
\begin{displaymath}
y_t=\xi-\int_t^1Z_{\phi^{-1}(s)}dW_{\phi^{-1}(s)}+\int_t^1f(\phi^{-1}(s),Y_{\phi^{-1}(s)},Z_{\phi^{-1}(s)})d\phi^{-1}(s)
\end{displaymath}

Now let
$h(s,\omega):=\biggl(\dfrac{d\phi^{-1}(s)}{ds}^{-\frac{1}{2}}\biggr)$
and $\widetilde{W}_t:=\D\int_0^t h(s,\omega)dW_{\phi^{-1}(s)}$,
then $\widetilde{W}_t$ is a $(\mathcal{\widetilde{F}}_t)-$
Brownian motion.

In fact, if we let $M_t:=W_{\phi^{-1}(t)}$, then as we have seen
before, $M_t\in\widetilde{\mathcal{M}}_2^{c,loc}$ and so
$\widetilde{W}_t\in\widetilde{\mathcal{M}}_2^{c,loc}$.

\noindent And since $<M>_t=\phi^{-1}(t)$,

\begin{displaymath}
<\widetilde{W}_t>=\int_0^t h^2(s,\omega)d<M>_s=\int_0^t
h^2(s,\omega) h^{-2}(s,\omega) ds= t
\end{displaymath}

\noindent So we have

\begin{displaymath}
y_t=\xi-\int_t^1Z_{\phi^{-1}(s)}h^{-1}(s,\omega)d\widetilde{W}_s+\int_t^1f(\phi^{-1}(s),Y_{\phi^{-1}(s)},Z_{\phi^{-1}(s)})h(s,\omega)^{-2}ds
\end{displaymath}

\noindent Since $z_t=Z_{\phi^{-1}(t)}h^{-1}(t,\omega)$, we have

\begin{displaymath}
y_t=\xi-\int_t^1 z_sd\widetilde{W}_s+\int_t^1 f(\phi^{-1}(s), y_s,
z_s\phi'(s)^{\frac{1}{2}})\phi'(s)^{-1}ds
\end{displaymath}

And $y_t,z_t$ are $(\widetilde{\mathcal{F}}_t)_{t\ge0}-$ adapted
because $Y_t, Z_t$ are $(\mathcal{F}_t)_{t\ge0}$ adapted and
$h(t,\omega)$ is $(\widetilde{\mathcal{F}}_t)_{t\ge0}-$ adapted.

So $(y_t,z_t)$ is a solution to the BSDE with constant time
horizon defined on $(\Omega, \mathcal{F}, \mathbb{P},
(\widetilde{\mathcal{F}}_t))$.
\end{proof}

\begin{remark}
Note that the transformation of Brownian motion is given as
following.
\begin{displaymath}
\begin{cases}
d\widetilde{W}_s=\phi'(s)^{\frac{1}{2}}dW_{\phi^{-1}(s)} \\
dW_s=\phi'(\phi(t))^{-\frac{1}{2}}d\widetilde{W}_{\phi^{-1}(s)} \\
\end{cases}
\end{displaymath}
\end{remark}

\begin{remark}
From now on,we'll use the notation $t\in[0,\tau]$ to represent the
set $\{(\omega,t)\in\Omega\times[0,\infty)\}$.
\end{remark}

\begin{remark}
From Lemma 2.1, we can consider the inverse mapping.

\noindent That is, if $(y_t,z_t)$ is a solution of (2.2),then
$(Y_t,Z_t)=(y_{\phi(t)}, z_{\phi(t)}\phi'(\phi(t))^\frac{1}{2})$
is a solution of (2.1). So in some sense, we can say that the two
equations (2.1) and (2.2) are equivalent and we can get results
for one equation from the results for the other one.
\end{remark}

\begin{example}
Recently there has been much work on quadratic BSDEs in
connections with partial differential equations. A typical and
very simple quadratic BSDE is
\begin{displaymath}
Y_t=\xi-\int_t^TZ_sdW_s+\int_t^T\alpha Z_s^2ds
\end{displaymath}
If we replace the time horizon with stopping time $\tau$ , we have
\begin{displaymath}
Y_t=\xi-\int_t^\tau Z_sdW_s+\int_t^\tau \alpha Z_s^2ds
\end{displaymath}

Here if we apply the transformation in the theorem we have
\begin{displaymath}
y_t=\xi-\int_t^1 z_sd\widetilde{W}_s+\int_t^1 \alpha z_s^2ds
\end{displaymath}
and it's just the same as the original one with constant time
horizon. From this example we can guess that there isn't any
difference between the two equations.
\end{example}

\section{Some Results}
In this section, we generalize some results known for the case of
constant time horizon to the case of random time horizon using our
method.

\subsection{Formula for the solution of linear BSDE}
The linear BSDEs are very important in mathematical finance and
there exists an explicit formula for the solution in the case of
constant time horizon. We generalize this formula to the case of
random time horizon.

\vskip 0.5cm
In this section we assume that the random time horizon
$\tau$ is bounded by real number $T>0$ and define the following
spaces:

\noindent
\begin{math}
\textrm{\qquad}\mathcal{P}_n : \textrm{the set of
}\mathcal{F}_t-\textrm{measurable, }\mathbb{R}^n-\textrm{valued
processes on }\Omega\times[0,T] \\
\textrm{\qquad}\mathcal
{L}_n^2(\mathcal{F}_t)=\{\eta:\mathcal{F}_t\textrm{-measurable
random } \mathbb R^n \textrm{-valued variable such that }
E[\lvert\eta\rvert^2]<\infty \} \\
\textrm{\qquad} \mathcal{S}_n^2=\{\varphi\in\mathcal{P}_n \textrm{
with continuous paths such that } E[\sup_{t\le
T}\lvert\varphi_t\rvert^2]<\infty \} \\
\textrm{\qquad} \mathcal{H}_n^2(0,T)=\{Z\in\mathcal{P}_n\;\lvert\;
E\D[\int_0^T|Z_s|^2ds]<\infty\} \\
\textrm{\qquad} \mathcal{H}_n^1(0,T)=\{Z\in\mathcal{P}_n\;\lvert\;
E\D[(\int_0^T|Z_s|^2ds)^{\frac{1}{2}}]<\infty\}
\end{math}
\theoremstyle{break}
\begin{theorem}[{El Karoui, N., Peng S., Quenez M.C.[10]}]$\\$
Let $(\beta, \mu)$ be a bounded $(\mathbb{R},\mathbb{R}^d)-$valued
progressively measurable process, $\varphi$ be an element of
$\mathcal{H}_1^2(0,T)$ and $\xi\in \mathcal{L}_1^2(0,T)$.

\noindent We consider the following linear BSDE:
\begin{displaymath}
\begin{cases}
-dY_t=(\varphi_t+Y_t\beta_t+Z_t\mu_t)dt-Z_tdW_t \\
Y_T=\xi
\end{cases}
\end{displaymath}
This equation has a unique solution
$(Y,Z)\in\mathcal{S}_1^2(0,T)\times\mathcal{H}_d^2(0,T)$ and $Y$
is given explicitly by
\begin{displaymath}
Y_t=E\biggl[\xi\Gamma_{t,T}+\int_t^T\Gamma_{t,s}\varphi_sds\;|\;\mathcal{F}_t\biggr]
\end{displaymath}
\noindent where $(\Gamma_{t,s})_{s\ge t}$ is the adjoint process
defined by the forward linear SDE
\begin{displaymath}
\begin{cases}
d\Gamma_{t,s}=\Gamma_{t,s}(\beta_sds+\mu_sdW_s)\\
\Gamma_{t,t}=1
\end{cases}
\end{displaymath}
\end{theorem}

\vskip 0.5cm
\begin{theorem}
Let $(\beta, \mu)$ be a bounded $(\mathbb{R},\mathbb{R}^d)-$valued
progressively measurable process, $\varphi$ be an element of
$\mathcal{H}_1^2(0,T)$ , $\tau$ be a $(\mathcal{F}_t)-$ stopping
time bounded by $T>0$ and $\xi\in \mathcal{L}_1^2(0,T)$.

\noindent We consider the following linear BSDE:
\begin{displaymath}
\begin{cases}
-dY_t=(\varphi_t+Y_t\beta_t+Z_t\mu_t)dt-Z_tdW_t \\
Y_\tau=\xi
\end{cases}
\end{displaymath}
This equation has a unique solution
$(Y,Z)\in\mathcal{S}_1^2(0,T)\times\mathcal{H}_d^2(0,T)$ and $Y$
is given explicitly by
\begin{displaymath}
Y_t=E\biggl[\xi\Gamma_{t,\tau}+\int_t^\tau\Gamma_{t,s}\varphi_sds\;|\;\mathcal{F}_t\biggr]
\end{displaymath}
\noindent where $(\Gamma_{t,s})_{s\ge t}$ is the adjoint process
defined by the forward linear SDE
\begin{displaymath}
\begin{cases}
d\Gamma_{t,s}=\Gamma_{t,s}(\beta_sds+\mu_sdW_s)\\
\Gamma_{t,t}=1
\end{cases}
\end{displaymath}
\end{theorem}
\begin{proof}
If we set $\phi(t,\omega):=\frac{t}{\tau}$ and apply the
transformation in Theorem 2.1, then we have
\begin{displaymath}
y_t=\xi+\int_t^1[\varphi_{s\tau}+\beta_{s\tau}y_s+\mu_{s\tau}z_s\frac{1}{\sqrt{\tau}}]\tau
ds-\int_t^1z_sd\widetilde{W}_s
\end{displaymath}
Then from the above theorem we get the explicit solution
\begin{displaymath}
y_t=E\biggl[\xi\widetilde{\Gamma}_{t,\tau}+\int_t^1\widetilde{\Gamma}_{t,s}\varphi_sds\;|\;\mathcal{\widetilde{F}}_t\biggr]
\end{displaymath}
And from Remark 2.2, the solution of the original equation is
given by
\begin{displaymath}
Y_t=E\biggl[\xi\Gamma_{t,\tau}+\int_t^\tau\Gamma_{t,s}\varphi_sds\;|\;\mathcal{F}_t\biggr]
\end{displaymath}
\noindent where $(\Gamma_{t,s})_{s\ge t}$ is the adjoint process
defined by the forward linear SDE
\begin{displaymath}
\begin{cases}
d\Gamma_{t,s}=\Gamma_{t,s}(\beta_sds+\mu_sdW_s)\\
\Gamma_{t,t}=1
\end{cases}
\end{displaymath}

\end{proof}

\subsection{Measure Solution}
Measure solutions of BSDE were introduced by S. Ankirchner, et
al.[16]and they prove the existence of measure solution for the
case of constant time horizon under some constraints. Using our
method, we can generalize the result to the case of random time
horizon. The generalization is not so difficult and we do not note
here. We just mention the result for the case of stopping time
horizon.

\vskip 0.5cm

Let $\tau$ be a $(\mathcal{F}_t)-$stopping time bounded by $T>0$,
$\xi$ be a bounded $\mathcal{F}_t-$-measurable variable,
$f:\Omega\times[0,T]\times\mathbb{R}\rightarrow\mathbb{R}$ be a
measurable function such that $f(\cdot, \cdot,z)$ is predictable
for all $z\in\mathbb{R}$ and satisfies following assumption.

{\bf Assumption(H):}

\qquad (i) $\forall z\in\mathbb{R},\;f(s,z)=f(\cdot,s,z)$ is
adapted.

\qquad (ii) $\forall
s\in[0,T],\;g(s,z)=\dfrac{f(s,z)}{z},\;\;z\in\mathbb{R}$ is
continuous in $z$

\qquad (iii) $\forall s\in[0,T],\forall z\in\mathbb{R},
|f(s,z)|\le c(1+z^2)$

\qquad (iv) there exists $\varepsilon>0$ and a predictable process
$(\psi_s)_{s\ge0}$ such that $\int_0^\cdot\psi_sdW_s$ is a
{\indent\hskip 1.5cm} BMO-martingale and for every
$|z|\le\varepsilon,\;|g(s,z)|\le\psi_s$.

Let $\xi$ be an $\mathcal{F}_\tau-$measurable random variable and
consider a BSDE

\begin{displaymath}
Y_t=\xi-\int_t^\tau Z_sdW_s+\int_t^\tau f(s,Z_s)ds\quad
t\in[0,\tau]
\end{displaymath}

\begin{theorem}
There exists a probability measure $\mathbb Q$ equivalent to
$\mathbb P$ and an adapted process $Z$ such that
$E(\bigl[\D\int_0^\tau Z_s^2ds\bigr]^{1}{2})<\infty$ such that,
setting
\begin{displaymath}
R=exp(\int_0^\tau g(s,Z_s)dW_s-\frac{1}{2}\int_0^\tau g(s,Z_s)^2ds
\\
W^\mathbb{Q}=W-\int_0^\cdot g(s,Z_s)ds
\end{displaymath}
\noindent we have $\mathbb{Q}=R\cdot\mathbb{P}$ and such that the
pair $(Y,Z)$ defined by
\begin{displaymath}
Y=E^\mathbb{Q}(\xi\;|\;\mathcal{F}_\cdot)=\int_0^\cdot
Z_sdW_s^\mathbb{Q}
\end{displaymath}
solves the BSDE
\begin{displaymath}
Y_t=\xi-\int_t^\tau Z_sdW_s+\int_t^\tau f(s,Z_s)ds\quad
t\in[0,\tau]
\end{displaymath}
\end{theorem}

\vskip0.5cm\noindent
{\bf References}
\small\vskip0.4cm
\begin{itemize}
\item[{[1]}] Ikeda, N., and Watanabe, S. : \emph {Stochastic
Differential Equations and Diffusion Processes}, North Holland,
Amsterdam, 1981

\item[{[2]}] Pardoux, E., Peng, S.: Adapted solution of a backward
stochastic differential equation. Systems Control Lett. 14(1),
55-61, 1990

\item[{[3]}] Kobylanski, M.: Backward stochastic differential
equations and partial differential equations with quadratic
growth. Ann. Probab. 28(2), 558-602, 2000

\item[{[4]}] Briand,P., Hu, Y: BSDE with quadratic growth and
unbounded terminal value. Probab. Theory Related Fields 136, 2006,
604-618

\item[{[5]}] Briand,P., Hu, Y: Quadratic BSDEs with convex
generators and unbounded terminal conditions.Probab. Theory Relat
Fields, 2008, 141:543-567

\item[{[6]}] Briand,P., Lepeltier, J.-P., San Martin, J.:
One-dimensional BSDE¡¯s whose coefficient is monotonic in $y$ and
non-lipschitz in $z$. Bernoulli 3(1), 80-91, 2007

\item[{[7]}] Briand, P., Delyon, B.,Hu, Y., Pardoux, E., Stoica,
L.: $L^p$ solutions of Backward Stochastic Differential Equations,
Stochastic Processes and their Applications, 108, 109-129, 2003

\item[{[8]}] Briand, P., Hu, Y.: BSDE with quadratic growth and
unbounded terminal value. Probab. Theory Relat Fields 136(4),
604-618, 2006

\item[{[9]}] Darling, R.W.R. and Pardoux, E.: Backward SDE with
random terminal time and applications to semilinear elliptic PDE.
Ann. Proba. Vol. 25, No.3, 1135-1159, 1997

\item[{[10]}] El Karoui, N., Peng S., Quenez M.C.: Backward
Stochastic Differential Equations in Finance,Mathematical Finance
7, 1-71,1997

\item[{[11]}] Blanchet-ScallietCh., Eyraud-LoiselA.,
Royer-CarenziM.: Hedging of Defaultable Contingent Claims using
BSDE with uncertain time horizon, preprint, 2009

\item[{[12]}] Lepeltier, J.-P., San Martin, J.: Existence for BSDE
with superlinear-quadratic coefficient. Stochastics Stochastics
Rep.63(3-4), 227-240, 1998

\item[{[13]}] Patrick Cheridito, H.Mete Soner, Nizar Touz and
Nicolas Victor: Second-Order Backward Stochastic Differential
Equations and Fully Nonlinear Parabolic PDEs, Communications on
Pure and Applied Mathematics, Vol. LX, 0001-0030, 2007

\item[{[14]}] Pardoux, E. and Peng, S.: BSDEs and quasilinear
parabolic PDEs. Lecture Notes in Control and Inform. Sci.
176:200-217. Springer, New York, 1992

\item[{[15]}] Barles, G., Buckdahn, R. and Pardoux, E.: Backward
stochastic differential equations and integral-partial
differential equations. Stochastics Stochastics Rep. 60 pp. 57-83,
1997

\item[{[16]}] S. Ankirchner, et al., On measure solutions of
backward stochastic differential equations, Stochastic Processes
and their Applications 119, 2744-2772, 2009

\end{itemize}
\end{document}